\documentstyle[11pt,amssymb]{article}

\newtheorem{Theo}{Theorem}
\newtheorem{Prop}{Proposition}
\newtheorem{Lemma}{Lemma}

\newcommand{\PP}{{\mathbb P}}
\newcommand{\EE}{{\mathbb E}}
\newcommand{\inr}{I_{n,r}}
\newcommand{\Mnr}{M_{n,r}}
\newcommand{\mnr}{m_{n,r}}

\newcommand{\CQFD}
{%
\mbox{}%
\nolinebreak%
\hfill%
\rule{2mm}{2mm}%
\medbreak%
\par%
}

\begin{document}

\title{A note on extreme values and kernel estimators
of sample boundaries}
\author{St\'ephane Girard$^{1,}$\footnote{Corresponding author}
~and Pierre Jacob$^{2}$}
\date{ INRIA Rh\^one-Alpes, team Mistis,\\
655, avenue de l'Europe, Montbonnot,\\
38334 Saint-Ismier Cedex, France. \\ 
{\tt Stephane.Girard@inrialpes.fr} \\
\vspace*{2mm}
$^2$ EPS/I3M,
Universit\'e Montpellier 2,\\
Place Eug\`ene Bataillon, 34095 Montpellier Cedex 5, France.
{\tt jacob@math.univ-montp2.fr}}

\maketitle

\begin{abstract}
In a previous paper~\cite{ESAIM}, we studied a kernel estimate of the upper
edge of a two-dimensional bounded set, based upon the extreme values of a
Poisson point process. The initial paper~\cite{G64}
on the subject  treats
the frontier as the boundary of the support set for a density
and the points as a random sample.
 We claimed in~\cite{ESAIM} that we are able to deduce the random sample case from the
point process case. 
The present note gives
some essential indications to this end, including a method
which can be of general interest.  \\

\end{abstract}

\section{Introduction and main results}
\label{secun}

As in the early paper of Geffroy~\cite{G64}, we address the problem of estimating a
subset $D$ of ${\mathbb R}^{2}$ given a sequence of random points $\Sigma
_{n}=\left\{  Z_{1},...,Z_{n}\right\}  $ where the $Z_{i}=(X_{i,}Y_{i})$ are
independent and uniformly distributed on $D$. The problem is reduced to
functional estimation by defining%
\[
D=\left\{  \left(  x,y\right)  \in{\mathbb R}^{2}/0\leq x\leq1;0\leq y\leq
f\left(  x\right)  \right\},
\]
where $f$ is a strictly positive function. Given an increasing sequence of
integers $0<k_n<n$, $k_{n}\uparrow\infty$, for $r=1,...,k_{n}$, let $I_{n,r}=[\left(
r-1\right)  /k_{n},r/k_{n}[$ and
\begin{equation}
\label{equn}
U_{n,r}=\max\left\{  Y_{i}/\left(  X_{i},Y_{i}\right)  \in\Sigma_{n};X_{i}\in 
I_{n,r}\right\},
\end{equation}
where it is conveniently understood that $\max \varnothing
=0$. Now, let $K$ be a bounded density, which has a support within a
compact interval $\left[  -A,A\right]$, a bounded first derivative and which
is piecewise $C^{2}$, and $h_{n}\downarrow0$ a sequence of positive numbers.
Following~\cite{ESAIM}, Section~6, we define the estimate
\begin{equation}
\label{eqdeux}
\hat{f}_{n}(x)=\frac{1}{k_{n}}\sum\limits_{r=1}^{k_{n}}K_{n}(x-x_{r})\left(
U_{n,r}+\frac{1}{n-k_{n}}\sum\limits_{s=1}^{k_{n}}U_{n,s}\right)
,x\in\mathbb{R},
\end{equation}
where $x_r$ is the center of $I_{n,r}$ and, as usually,
$$
K_{n}\left(  t\right)  =\frac{1}{h_{n}}K\left(  \frac{t}{h_{n}}\right)
,t\in\mathbb{R}.
$$
The perhaps curious second term in brackets in formula~(\ref{eqdeux}) is designed for
reducing the bias (see~\cite{ESAIM}, Lemma~8). Note that $\hat f_n$
can be rewritten as a linear combination of extreme values
$$
\hat{f}_{n}(x)=\frac{1}{k_n}\sum_{r=1}^{k_n} \beta_{n,r}(x) U_{n,r},
$$
where 
$$
\beta_{n,r}(x)= \frac{1}{k_{n}} K_{n}(x-x_{r}) + \frac{1}{k_n(n-k_{n})}
\sum\limits_{s=1}^{k_{n}} K_n(x-x_s).
$$
In the sequel, we suppose 
that $f$ is $\alpha$-Lipschitzian, $0<\alpha\leq1$, and strictly positive.
Our result is the following:

\begin{Theo}
\label{thmain1}
If $h_{n}k_{n}\rightarrow\infty$, $n=o(  k_{n}^{1/2}h_{n}^{-1/2-\alpha})$,
$n=o(k_{n}^{5/2}h_{n}^{3/2})$ and $k_{n}=o(n/\ln n)$,
then for every $x\in\left]  0,1\right[  ,$%
$$
\left(  nh_{n}^{1/2}/k_{n}^{1/2}\right)  \left(  \hat{f}_{n}\left(  x\right)
-f\left(  x\right)  \right)  \Rightarrow{\mathcal N}\left(  0,\sigma ^{2}\right),
$$
with $\sigma=\|K\|_2/c$.  
\end{Theo}

\section{Proofs}
\label{secdeux}
If formally the definition of $\hat{f}_{n}$ is identical here and in~\cite{ESAIM}
the fundamental difference lies in the fact that in~\cite{ESAIM} the sample
is replaced by a homogeneous Poisson point process with a mean measure
$\mu_{n}=nc\lambda 1_D$ where $\lambda$ is the Lebesgue measure of ${\mathbb R}^{2}$
and $c^{-1}=\lambda(D)$. Here we denote by $\Sigma_{0,n}$ this point
process and we need, for the sake of approximation, two further Poisson point
processes $\Sigma_{1,n}$ and $\Sigma_{2,n}$. 
The point processes $\Sigma_{j,n}$ are constructed as in~\cite{NONPAR}, extending an
original idea of J.~Geffroy. Given a sequence $\gamma_{n}\downarrow0$, consider
independent Poisson random variables $N_{1,n},\ M_{1,n},\ M_{2,n}$,
independent of the sequence $\left(  Z_{n}\right)  $, with parameters
$\EE\left(  N_{1,n}\right)  =n(1-\gamma_{n})$ and $\EE\left(  M_{1,n}\right)
=\EE\left(  M_{2,n}\right)  =n\gamma_{n}$. Define $N_{0,n}=N_{1,n}+M_{1,n}$, 
$N_{2,n}=N_{0,n}+M_{2,n}$ and take $\Sigma_{j,n}=\left\{  Z_{1}%
,...,Z_{N_{j,n}}\right\}$, $j=0,1,2$. 
For $j=0,1,2$ we define
$U_{j,n,r}$ and $\hat{f}_{j,n}$ by imitating~(\ref{equn}) and (\ref{eqdeux}).
Finally, let us introduce the event $E_{n}=\{\Sigma_{1,n}\subseteq\Sigma_{n}\subseteq\Sigma_{2,n}\}$. The following lemma is the starting point of
our "random sandwiching" technique.
\begin{Lemma}
\label{lemsand}
One always has $\hat {f}_{1,n}\leq\hat{f}_{0,n}\leq\hat{f}_{2,n}$.
Moreover, if $E_n$ holds,
$\hat{f}_{1,n}\leq\hat{f}_{n}\leq\hat{f}_{2,n}$.
\end{Lemma}
\begin{proof}
The definition of the random sets $\Sigma_{j,n}$, $j=0,1,2$ implies that
$\Sigma_{1,n}\subseteq\Sigma_{0,n}\subseteq\Sigma_{2,n}$.
Thus, since $\beta_{n,r}(x)\geq 0$ for all $r=1,\dots,k_n$,
we have $\hat {f}_{1,n}\leq\hat{f}_{0,n}\leq\hat{f}_{2,n}$.
Similarly, $E_n$ implies that $\hat{f}_{1,n}\leq\hat{f}_{n}\leq\hat{f}_{2,n}$.
\CQFD
\end{proof}
The success of the approximation between $\hat{f}_{n}$ and $\hat{f}_{0,n}$ is
based upon two lemmas. The first one shows how large is the probability of the
event $E_n$.
\begin{Lemma}
\label{lemun}
For $n$ large enough,
$$
\PP\left(  \Omega\smallsetminus E_{n}\right) \leq
2 \exp\left(-\frac{1}{8}n\gamma_n^2\right) .
$$
\end{Lemma}
\begin{proof}
Using the Laplace transform of a Poisson random variable $X$ with
parameter $\lambda>0$, we get for $\varepsilon/2\lambda$ small enough,
$$
\PP\left(  |X-\lambda|>\varepsilon\right)  <\exp\left(  -\varepsilon^{2}
/4\lambda\right),
$$
see for instance Lemma~1 in~\cite{NONPAR}.
 Clearly, $\Omega\smallsetminus E_{n}=\left\{
N_{1,n}>n\right\}  \cup\left\{  N_{2,n}<n\right\}  $ and thus
$$
\PP\left(  \Omega\smallsetminus E_{n}\right) \leq
 \exp\left(-\frac{n\gamma_n^2}{4(1-\gamma_n)}\right) +
 \exp\left(-\frac{n\gamma_n^2}{4(1+\gamma_n)}\right).
$$
The lemma follows.\CQFD
\end{proof}
The second lemma is essential to control the approximation obtained when the
event $E_{n}$ holds.

\begin{Lemma}
\label{lemdeux}
If $k_{n}=o(n/\log n)$ and $n=O\left(  k_{n}^{1+\alpha}\right)  $, then
uniformly on $r=1,\dots,k_n$,
$$
\EE\left(U_{2,n,r}-U_{1,n,r}\right)  = O\left({\frac{k_n\gamma_n}{n}}\right).
$$
\end{Lemma}
\begin{proof}
Let us define $\mnr=\displaystyle\min_{x\in\inr} f(x)$ and $\displaystyle\Mnr=\max_{x\in\inr}f(x)$.  Then,
\begin{eqnarray*}
&&\EE\left(  U_{2,n,r}-U_{1,n,r}\right)  \\
&=& \int_{0}^{M_{n,r}}\left(  \PP\left( U_{2,n,r}>y\right)  -\PP\left(  U_{1,n,r}>y\right)  \right)  dy\\
&  =&\int_{0}^{m_{n,r}} \left(  \PP\left( U_{2,n,r}>y\right)  -\PP\left(  U_{1,n,r}>y\right)  \right)  dy 
+\int_{m_{n,r}}^{M_{n,r}} \left(  \PP\left( U_{2,n,r}>y\right)  -\PP\left(  U_{1,n,r}>y\right)  \right)  dy\\
&\stackrel{def}{=}&A_{n,r}+B_{n,r}.  
\end{eqnarray*}
Introducing $\lambda_{n,r}=\int_{I_{n,r}}f(x)dx,$ we can write $A_{n,r}$ as%
\[
A_{n,r}=\int_{0}^{m_{n,r}}\exp\left(  \frac{n(1-\gamma_{n})}{k_{n}}\left(
y-k_{n}\lambda_{n,r}\right)  \right)  dy-\exp\left(  \frac{n(1+\gamma_{n}%
)}{k_{n}}\left(  y-k_{n}\lambda_{n,r}\right)  \right)  dy.  
\]
Now, $A_{n,r}$ is expanded as a sum $A_{1,n,r}+A_{2,n,r}$ with%
\begin{eqnarray*}
A_{1,n,r}  &  =&\frac{k_{n}}{n(1-\gamma_{n})}\exp\left(  \frac{n(1-\gamma_{n}%
)}{k_{n}}\left(  m_{n,r}-k_{n}\lambda_{n,r}\right)  \right)  -\frac{k_{n}%
}{n(1+\gamma_{n})}\exp\left(  \frac{n(1+\gamma_{n})}{k_{n}}\left(
m_{n,r}-k_{n}\lambda_{n,r}\right)  \right), \\
A_{2,n,r}  &  =&\frac{k_{n}}{n(1+\gamma_{n})}\exp\left(  -n(1+\gamma_{n}%
)\lambda_{n,r}\right)  -\frac{k_{n}}{n(1-\gamma_{n})}\exp\left(
-n(1-\gamma_{n})\lambda_{n,r}\right).
\end{eqnarray*}
The part $A_{2,n,r}$ is easily seen to be a $o\left(  n^{-s}\right)  $ where $s$
is a arbitrarily large exponent under the condition $k_{n}=o(n/\log n)$. Now,
If $a,b,x,y$ are real numbers such that $x<y<0<b<a$, we have \ $0<ae^{y}%
-be^{x}<\left(  a-b\right)  +b\left(  y-x\right)  $. Applying to $A_{1,n,r}$ this
inequality yields
\[
A_{1,n,r}\leq\frac{k_{n}}{n}\frac{2\gamma_{n}}{\left(  1-\gamma_{n}^{2}\right)
}+\left(  M_{n,r}-m_{n,r}\right)  \frac{2\gamma_{n}}{\left(  1+\gamma
_{n}\right)  }.
\]
Under the hypothesis that $f$ is $\alpha-$Lipschitzian, 
and the condition $n=O\left(  k_{n}^{1+\alpha}\right)$,  
we have $\left(  M_{n,r}-m_{n,r}\right)
=O({k_{n}}/{n})$, so that $A_{n,r}=A_{1,n,r}+A_{2,n,r}=O\left(  {k_{n}}
\gamma_{n}/n\right)  $.
Now, for $m_{n,r}\leq y\leq M_{n,r}$, it is easily seen that%
\[
\PP\left(  U_{2,n,r}>y\right)  -\PP\left(  U_{1,n,r}>y\right)  \leq2\gamma
_{n}\frac{n}{k_{n}}(M_{n,r}-m_{n,r}),
\]
and thus
$$
B_{n,r}\leq2\gamma_{n}\frac{n}{k_{n}}\left(  M_{n,r}-m_{n,r}\right)
^{2}=O\left(  \frac{k_{n}}{n}\gamma_{n}\right).
$$
Clearly, the bounds on
$A_{n,r}$ and $B_{n,r}$ are uniform in $r=1,...,k_{n}$, and thus we obtain the result.
\CQFD
\end{proof}
We quote a technical lemma.
\begin{Lemma}
\label{lemtrois}
If $k_{n}=o\left(  n\right)  $ and $h_{n}k_{n}\rightarrow\infty$ when
$n\rightarrow\infty,$%
$$
\lim_{n\rightarrow\infty} \sum_{r=1}^{k_n} \beta_{n,r}(x) = 1.
$$
\end{Lemma}
\begin{proof}
Remarking that 
$$
\sum_{r=1}^{k_n} \beta_{n,r}(x) = \frac{n}{n-k_n} \frac{1}{k_n}
\sum_{r=1}^{k_n} K_{n}(x-x_{r}),
$$
the result follows from the well-known property
$$
\lim_{n\rightarrow\infty}\frac{1}{k_{n}}\sum\limits_{r=1}^{k_{n}%
}K_{n}(x-x_{r})=1,
$$
see for instance~\cite{ESAIM}, Corollary~2.
\CQFD
\end{proof}
The next proposition is the key tool to extend the results
obtained on Poisson processes to samples.  
\begin{Prop}
\label{prop1}
If $k_{n}=o(n/\log n)$, $h_nk_n\to\infty$, and 
$n=O\left(k_{n}^{1+\alpha}\right)$, then, for
every $x\in\left]  0,1\right[  $,
$$
(nh_{n}^{1/2}/k_{n}^{1/2})\EE\left(\left\vert \hat{f}_{n}(x)-\hat{f}_{0,n}%
(x)\right\vert \right)\rightarrow0.
$$
\end{Prop}
\begin{proof}
From Lemma~\ref{lemsand}, we have
\begin{eqnarray*}
\EE\left(  \left\vert \hat{f}_{n}(x)-\hat{f}_{0,n}(x)\right\vert {\mathbf 1}%
_{E_{n}}\right) 
&\leq& \EE\left(\hat{f}_{2,n}(x)-\hat{f}_{1,n}(x)\right)\\
&=&\sum\limits_{r=1}^{k_{n}}\beta_{n,r}(x) \EE(U_{2,n,r}-U_{1,n,r})\\
&\leq &\sum\limits_{r=1}^{k_{n}}\beta_{n,r}(x) \max_{1\leq s \leq k_n}
\EE(U_{2,n,s}-U_{1,n,s})\\
&=&O\left(\frac{k_{n}\gamma_n}{n}\right),
\end{eqnarray*}
in view of Lemma~\ref{lemdeux} and Lemma~\ref{lemtrois}.
As a consequence,
\begin{equation}
\label{eqtmp1}
(nh_{n}^{1/2}/k_{n}^{1/2})\EE\left(\left\vert \hat{f}_{n}(x)-\hat
{f}_{0,n}(x)\right\vert {\mathbf 1}_{E_{n}}\right)
=O\left(  k_{n}^{1/2}h_{n}^{1/2}\gamma_{n}\right).
\end{equation}
Now, let $M=\sup\left\{  f\left(  x\right)  ,x\in\left[  0,1\right]  \right\}$.
Then, applying Lemma~\ref{lemtrois} again,
$$
\max\left\{  \hat{f}_{n}(x),\hat{f}_{0,n}(x)\right\}  \leq
M \sum\limits_{r=1}^{k_{n}}\beta_{n,r}(x)=O(1),
$$
and therefore, from Lemma~\ref{lemun},
\begin{eqnarray}
\label{eqtmp2}
(nh_{n}^{1/2}/k_{n}^{1/2})\EE\left(  \left\vert \hat{f}_{n}(x)-\hat{f}%
_{0,n}(x)\right\vert {\mathbf 1}_{\Omega\smallsetminus E_{n}}\right)
&=& (nh_{n}^{1/2}/k_{n}^{1/2})O(1)\PP\left(  \Omega\smallsetminus E_{n}\right)\nonumber\\
&  =&o(n)\exp\left(-\frac{1}{8}n\gamma_{n}^{2}\right)\nonumber\\
&  =&o(1)\exp\left(-\frac{n}{k_{n}}\left(\frac{1}{8}k_{n}\gamma_{n}^{2}-\frac{k_{n}}
{n}\log n\right)\right).
\end{eqnarray}
From (\ref{eqtmp1}) and (\ref{eqtmp2}) it suffices to take $\gamma_{n}=k_{n}^{-1/2}$ to obtain
the desired result.  
\CQFD
\end{proof}
The main theorem is now obtained without difficulty.
\paragraph{Proof of Theorem~\ref{thmain1}.}
Under the conditions $h_{n}k_{n}\rightarrow\infty$, $n=o(k_{n}^{1/2}h_{n}^{-1/2-\alpha })$, $n=o(k_{n}^{5/2}h_{n}^{3/2})$ and $k_{n}=o(n/\ln n)$, Theorem~5 of~\cite{ESAIM} asserts that%
\[
\left(  nh_{n}^{1/2}/k_{n}^{1/2}\right)  \left(  \hat{f}_{0,n}\left(
x\right)  -f\left(  x\right)  \right)  \Rightarrow    {\mathcal N}\left(
0,\sigma^{2}\right),
\]
while from Proposition~\ref{prop1},
\[
\left(  nh_{n}^{1/2}/k_{n}^{1/2}\right)  \left(  \hat{f}_{0,n}\left(
x\right)  -\hat{f}_{n}\left(  x\right)  \right)  \stackrel{\PP}{\rightarrow}0.
\]
Thus, the result  is an immediate application of Slutsky's theorem.
\CQFD

\end{document}